\documentclass[11pt]{amsart}

\usepackage{amsfonts}
\usepackage{amsthm}
\usepackage{amssymb}
\usepackage{amsmath}
\usepackage{url, hyperref}

\newtheorem{theorem}{Theorem}[section]
\newtheorem{lemma}[theorem]{Lemma}
\newtheorem{corollary}[theorem]{Corollary}

\newtheorem{proposition}[theorem]{Proposition}
\newtheorem{definition}[theorem]{Definition}

\newcommand{\Z}{{\mathbb Z}}
\newcommand{\N}{{\mathbb N}}

\newcommand{\R}{{\mathbb R}}

\newcommand{\bea}{\begin{eqnarray*}}
\newcommand{\eea}{\end{eqnarray*}}
\newcommand{\be}{\begin{eqnarray}}
\newcommand{\ee}{\end{eqnarray}}

\newcommand{\vol}{\mathrm{vol}\,}

\newcommand{\prob}{\mbox{\rm Prob}\,}
\newcommand{\frob}{\mbox{\rm F}}

\numberwithin{equation}{section}

\begin{document}

\title[Generalized Frobenius numbers]{Generalized Frobenius numbers:\\Bounds and average behavior}

\author{Iskander Aliev}
\address{School of Mathematics and Wales Institute of Mathematical and Computational Sciences, Cardiff University, Senghennydd Road, Cardiff, Wales, UK}
\email{alievi@cf.ac.uk}

\author{Lenny Fukshansky}
\address{Department of Mathematics, 850 Columbia Avenue, Claremont McKenna College, Claremont, CA 91711, USA}
\email{lenny@cmc.edu}

\author{Martin Henk}
\address{Fakult\"at f\"ur Mathematik, Otto-von-Guericke
Universit\"at Mag\-deburg, Universit\"atsplatz 2, D-39106-Magdeburg}
\email{martin.henk@ovgu.de}


\begin{abstract}
Let $n \geq 2$ and  $s \geq 1$ be integers and $a = (a_1,\dots,a_n)$ be a relatively prime integer $n$-tuple. The $s$-Frobenius number of this $n$-tuple, $\frob_s(a)$, is defined to be the largest positive integer that cannot be represented as $\sum_{i=1}^n a_i x_i$ in at least $s$ different ways, where $x_1,...,x_n$ are non-negative integers. This natural generalization of the classical Frobenius number, $\frob_1(a)$, has been studied recently by a number of authors. We produce new upper and lower bounds for the $s$-Frobenius number by relating it to the so called $s$-covering radius of a certain convex body with respect to a certain lattice; this generalizes a well-known theorem of R. Kannan for the classical Frobenius number. Using these bounds, we obtain results on the average behavior of the $s$-Frobenius number, extending analogous recent investigations for the classical Frobenius number by a variety of authors. We also derive bounds on the $s$-covering radius, an interesting geometric quantity in its own right.
\end{abstract}

\maketitle

\section{Introduction}
\label{intro}

Let $a$ be a positive integral $n$-dimensional primitive vector,
i.e., $a=(a_1,\dots,a_n)^\intercal\in\Z_{>0}^n$ with
$\gcd(a):=\gcd(a_1,\dots,a_n)=1$, so that $a_1 < a_2 < \dots < a_n$. For a positive integer $s$ the
{\em $s$-Frobenius number} $\frob_s(a)$, is the largest number which cannot be
represented in at least $s$ different ways as a non-negative integral combination of the $a_i$'s, i.e.,
\begin{equation*}
  \frob_s(a)=\max\{b\in\Z : \#\{z\in\Z^n_{\geq 0} : \langle a, z\rangle=b\}<s\},
\end{equation*}
where $\langle \cdot, \cdot\rangle$ denotes the standard inner product
on $\R^n$.

This generalized Frobenius number has been introduced and studied by
Beck and Robins \cite{BeckRobins:extension}, who showed, among other results, that for $n=2$
\begin{equation}
      \frob_s(a)=s\,a_1\,a_2-(a_1+a_2).
\label{eq:dim_two_ext}
\end{equation}
In particular, this identity generalizes the well-known result  in the setting of the (classical) Frobenius number
which corresponds to $s=1$. The origin of this classical result is unclear, it was most likely known already to Sylvester, see e.g. \cite{Sylvester}.
The literature on the Frobenius number $\frob_1(a)$ is vast; for a comprehensive and extensive
survey  we refer the reader to the book of Ramirez Alfonsin \cite{RamirezAlfonsin:2005p6858}.
Despite the exact formula in the
case $n=2$, for general $n$ only bounds on the Frobenius
number $\frob_1(a)$ are available. For instance, for $n\geq 3$
\begin{equation}
\left((n-1)!\,a_1\cdot\ldots\cdot a_n\right)^\frac{1}{n-1}-\left(a_1+\cdots +a_n\right)
<\frob_1(a)\leq 2\,a_n\,\left[\frac{a_1}{n}\right]-a_1.
\label{eq:bounds_classic}
\end{equation}
Here the lower bound follows from a sharp lower bound due to Aliev and Gruber \cite{Aliev:2007p6944}, and the
upper bound is due to Erd\H{o}s and Graham \cite{ErdHos:1972p6594}.  Hence, in the
worst case scenario we have an upper bound of the order $|a|_\infty^2$ on
the Frobenius number with respect to the maximum norm of the input
vector $a$. It is worth a mention that an upper bound on $\frob_1(a)$, which is symmetric
in all of the $a_i$'s has recently been produced by Fukshansky and Robins \cite{Fukshansky:2007p6629}. The quadratic order of the upper bound is known to be optimal (see, e.g., \cite{ErdHos:1972p6594})
and in view of the lower bound which is at most of size $|a|_\infty^\frac{n}{n-1}$
 it is quite natural to study the average behavior  of
 $\frob_1(a)$. This research was initiated and strongly influenced by Arnold
 \cite{Arnold:1999p6499}--\cite{Arnold:2006p151}, and due to recent results of Bourgain and Sinai \cite{Bourgain:2007p560},  Aliev and Henk \cite{AlievHenk}, Aliev, Henk and Hinrichs \cite{AHH},
 Marklof \cite{Marklof}, Li \cite{Li}, Shur, Sinai and Ustinov \cite{SSU},  Str\"ombergsson \cite{Str} and Ustinov \cite{Ustinov}
 we have a pretty clear picture of ``the average Frobenius number''. In order
 to describe some of these results, which are going to extend to the
 $s$-Frobenius number $\frob_s(a)$, we need a bit more notation. Let
\begin{equation*}
   \mathrm{G}(T)=\{ a\in\Z^n_{>0}: \gcd(a)=1,\,\vert a\vert_\infty\leq
   T\},
\end{equation*}
be the set of all possible input vectors of the Frobenius problem of
size (in maximum norm) at most $T$. Aliev, Henk and Hinrichs \cite{AHH} showed that
\begin{equation}
 \sup_T  \frac{\sum_{a\in
     \mathrm{G}(T)}\frob_1(a)/\left(a_1\,a_2\cdot\ldots\cdot
     a_n\right)^{\frac{1}{n-1}}}{\#\mathrm{G}(T) }\ll \gg_n1,
\label{eq:average_classic}
\end{equation}
i.e., the expected size of $\frob_1(a)$ is ``close'' to the size of its  lower
bound in~\eqref{eq:bounds_classic}; here and below $\ll_n$ and $\gg_n$
denote the Vinogradov symbols with the constant depending on $n$ only.
Recently, Li \cite{Li} gave the bound
\begin{equation}
 \prob\left(\frob_1(a)/\left(a_1\,a_2\cdot\ldots\cdot
     a_n\right)^{\frac{1}{n-1}} \geq D \right) \ll_n D^{-(n-1)},
\label{eq:prob_classic}
\end{equation}
where $\prob(\cdot)$ is meant with respect to the uniform
distribution among all points in the set $\mathrm{G}(T)$.
The  bound (\ref{eq:prob_classic}) is best possible due to an unpublished
result of Marklof, and clearly implies (\ref{eq:average_classic}).

The main purpose of this paper is to extend the results stated
above, i.e., \eqref{eq:bounds_classic}, \eqref{eq:average_classic} and
\eqref{eq:prob_classic}, to the generalized Frobenius number
$\frob_s(a)$ in the following way:

\begin{theorem} Let $n\geq 2$, $s\geq 1$. Then
\begin{equation*}
\begin{split}
\frob_s(a)&\geq
s^\frac{1}{n-1}\,\left((n-1)!\,a_1\cdot\ldots\cdot
  a_n\right)^\frac{1}{n-1}-\left(a_1+\cdots +a_n\right),\\
\frob_s(a)&\leq
\frob_1(a) + (s-1)^\frac{1}{n-1}\,\left((n-1)!\,a_1\cdot\ldots\cdot a_n\right)^\frac{1}{n-1}.
\end{split}
\end{equation*}
\label{thm:main}
\end{theorem}

\noindent
Bounds with almost the same dependencies on $s$ were recently obtained by
Fukshansky and Sch\"urmann \cite{FS}. Their lower bound, however, is only valid
for sufficiently large $s$.
Aliev and Gruber \cite{Aliev:2007p6944} applied the results of Schinzel \cite{NewSL} to obtain a sharp
lower bound for the Frobenius number in terms of the covering radius of a simplex. The same approach can be used to obtain a sharp lower bound for the $s$-Frobenius number as well. We postpone a detailed discussion of these matters to a future paper.

As an almost immediate consequence of Theorem \ref{thm:main} we obtain:

\begin{corollary} Let $n\geq 3$, $s\geq 1$. Then
\begin{equation*}
\begin{split}
{\rm i)}&\quad\prob\left(\frob_s(a)/\left(s\cdot a_1\,a_2\cdot\ldots\cdot
     a_n\right)^{\frac{1}{n-1}} \geq D \right) \ll_n D^{-(n-1)},\\
{\rm ii)}&\quad \sup_T  \frac{\sum_{a\in
     \mathrm{G}(T)}\frob_s(a)/\left(s\cdot a_1\,a_2\cdot\ldots\cdot
     a_n\right)^{\frac{1}{n-1}}}{\#\mathrm{G}(T) }\ll \gg_n1.
\end{split}
\end{equation*}
\label{cor:main}
\end{corollary}

\noindent
Hence in this generalized setting the average $s$-Frobenius number is of
the size $\left(s\cdot a_1\,a_2\cdot\ldots\cdot
  a_n\right)^{\frac{1}{n-1}}$, which again is the size of its lower
bound as stated in Theorem \ref{thm:main}.

The proof of Theorem~\ref{thm:main} is based on a generalization of a result of
Kannan which relates the classical Frobenius number to the covering radius of a
certain simplex with respect to a certain lattice. In our setting we
need a kind of generalized covering radius, whose definition as well
as  some properties and  background information from the
Geometry of Numbers will be given in Section~\ref{s-cover}. In Section~\ref{kannan}  we will
prove, analogously to the mentioned result of Kannan, an identity  between
$\frob_s(a)$ and this generalized covering radius and will present a proof
of Theorem~\ref{thm:main}. The last section contains a proof of
Corollary~\ref{cor:main}.
\bigskip

\section{The $s$-covering radius}
 \label{s-cover}

In what follows, let $\mathcal{K}^n$ be the space of all full-dimensional convex
bodies, i.e., closed bounded convex sets with non-empty interior in
the $n$-dimensional Euclidean space $\R^n$. The volume of a set
$X\subset\R^n$, i.e., its $n$-dimensional
Lebesgue measure, is denoted by $\vol X$.
Moreover, we denote by
$\mathcal{L}^n$ the set of all $n$-dimensional lattices in $\R^n$,
i.e., $\mathcal{L}^n=\{B\,\Z^n : B\in\R^{n\times n},\,\det
B\ne 0\}$. For $\Lambda=B\,\Z^n\in\mathcal{L}^n$, $\det\Lambda=|\det B|$ is
called the determinant of the lattice $\Lambda$.  Here we are
interested in the following quantity:

\begin{definition} Let $s\in\N$, $s\geq 1$. For $K\in\mathcal{K}^n$
  and $\Lambda\in\mathcal{L}^n$ let
\begin{equation*}
\begin{split}
 \mu_s(K,\Lambda)  =\min\{\mu > 0 : &\text{ for all } t\in\R^n \text{ there
   exist } b_1,\dots,b_s\in\Lambda \\ &\text{ such that } t\in
 b_i+\mu K\ \forall\ 1\leq i\leq s\}
\end{split}
\end{equation*}
be the smallest
positive number $\mu$ such that any $t\in\R^n$ is covered by at least $s$
lattice translates of $\mu\, K$. $\mu_s(K,\Lambda)$ is called the
\emph{$s$-covering radius} of $K$ with respect to $\Lambda$.
\end{definition}
For $s=1$ we get the well-known covering radius, for the information about which we refer the reader to
Gruber \cite{peterbible} and Gruber and Lekkerkerker \cite{GL}. These books also serve as excellent sources for more
 information on lattices and convex bodies in the context
of Geometry of Numbers.

Note that the $s$-covering radius is different from the $j$th covering minimum introduced by Kannan and Lov\'asz \cite{Kannan-Lovasz}.
We also remark that $\mu_s(K,\Lambda)$ may be described
equivalently as  the smallest positive number $\mu$ such that any translate of $\mu K$
contains at least $s$ lattice points, i.e.,
\begin{equation}
 \mu_s(K,\Lambda)=\min\{\mu>0 : \#\{ (t+\mu K)\cap\Lambda\}\geq s\text{
   for all }t\in\R^n\}.
\label{eq:s-covering}
\end{equation}

\begin{lemma} Let $s\in\N$, $s\geq 1$,  $K\in\mathcal{K}^n$
  and let $\Lambda\in\mathcal{L}^n$. Then
\begin{equation*}
       s^\frac{1}{n}\left(\frac{\det\Lambda}{\vol K}\right)^\frac{1}{n}\leq
       \mu_s(K,\Lambda) \leq  \mu_1(K,\Lambda) + (s-1)^\frac{1}{n}\left(\frac{\det\Lambda}{\vol K}\right)^\frac{1}{n}.
\end{equation*}
\label{lem:s-covering}
\end{lemma}

\begin{proof} It suffices to prove these inequalities for the standard lattice
  $\Z^n$ of  determinant 1; for brevity, we will just write
  $\mu_s$ instead of $\mu_s(K,\Z^n)$. The lower bound just reflects
  the fact that each point of $\R^n$ is covered at least $s$ times by
  the lattice translates of $\Z^n+\mu_s\,K$. A standard argument to
  see this in a more precise way is the following. Let
  $P=[0,1)^n$ be the half open cube of edge length $1$, and for
  $L \subseteq\R^n$ let $\chi_L : \R^n\to\{0,1\}$ be its characteristic
  function, i.e., $\chi_L(x)=1$  if $x\in L$, otherwise it is
  $0$. Then with $L=\mu_s\,K$ we get
\begin{equation}
\label{eq:vander}
\begin{split}
\vol(L)&=\int_{\R^n}\chi_{L}(x)\,\mathrm{d}x =
\int_{\Z^n+P}\chi_{L}(x)\,\mathrm{d}x
=\sum_{z\in\Z^n} \int_{z+P}\chi_{L}(x)\,\mathrm{d}x\\
&=\sum_{z\in\Z^n}\int_{P}\chi_{-z+L}(x)\,\mathrm{d}x
=\int_P\left(\sum_{z\in\Z^n}\chi_{-z+L}(x)\right)
\,\mathrm{d}x\\&\geq \int_P s
\,\mathrm{d}x =s.
\end{split}
\end{equation}
Hence $\vol(\mu_s\,K)\geq s$. Combining this observation with
the homogeneity of the volume we obtain the lower
bound. For the upper bound we may assume $s\geq 2$, since there
is nothing to prove for $s=1$. The first two lines of \eqref{eq:vander} also
prove a well-known result of van der Corput \cite[pp. 47]{GL}, which in
our setting of a convex body says: if $L\in\mathcal{K}^n$ with $\vol(L)\geq s-1$ then
there exists a $t\in P$ such that $t$ is covered by at least $s$
lattice translates of $L$. Hence for $\overline{\mu}= ((s-1)/\vol(K))^{1/n}$
we know that there exist   $z_1,\dots, z_s\in \Z^n$ and a $\overline{t}\in P$
such that $\overline{t}\in z_i+\overline{\mu}K$, $1\leq i\leq s$. Now
given an arbitrary $t\in\R^n$ we know by the definition of the
covering radius $\mu_1$ that there exists a $z\in\Z^n$ such that
$t-\overline{t}\in z+\mu_1\,K$. Hence
\begin{equation*}
  t\in (z+z_i) + (\mu_1+\overline{\mu})\,K,\quad 1\leq i\leq s,
\end{equation*}
 and so $\mu_s\leq \mu_1+\overline{\mu}$ which gives the upper bound.
\end{proof}

It is also worth a mention that, as an immediate corollary of
Lemma~\ref{lem:s-covering} and tools from the Geometry of Numbers, we
can obtain upper bounds on $\mu_s(K,\Lambda)$ for any $s \geq 1$ in
terms of successive minima of $K$ with respect to $\Lambda$. Recall
that successive minima $\lambda_i(K,\Lambda)$  of a convex body $K\in\mathcal{K}^n$ with
respect to a lattice $\Lambda\in\mathcal{L}^n$ are defined by
$$\lambda_i(K,\Lambda) = \min \left\{ \lambda > 0 : \dim \left(
    \lambda (K-K) \cap \Lambda \right) \geq i \right\}, 1\leq i\leq n.
$$

\begin{proposition}  Let $s\in\N$, $s\geq 1$,  $K\in\mathcal{K}^n$
  and let $\Lambda\in\mathcal{L}^n$. Then
\begin{equation*}
     \mu_s(K,\Lambda)\leq \left(1+(n!^{1/n}/n) (s-1)^{1/n}\right)\sum_{i=1}^n\lambda_i(K,\Lambda).
\end{equation*}
\end{proposition}
\proof It was pointed out  by Kannan and Lovasz \cite[Lemma 2.4]{Kannan-Lovasz}
that Jarnik's inequalities, relating the covering radius
and the successive minima of $0$-symmetric convex bodies, are also
valid for arbitrary bodies. Hence we have
\begin{equation}
  \mu_1(K,\Lambda)\leq \sum_{i=1}^n\lambda_i(K,\Lambda).
\label{eq:jarnik_up}
\end{equation}
On the other hand it is also well known that Minkowski's theorems on
successive minima can also be extended to the family of arbitrary
convex bodies \cite[pp.~59]{GL}, \cite{Henze},
and, in particular, we have
\begin{equation}
  \vol(K) \prod_{i=1}^n \lambda_i(K,\Lambda) \geq
  \frac{1}{n!}\,\det\Lambda.
\label{eq:minkowski_lower}
\end{equation}
Applying \eqref{eq:jarnik_up} and  \eqref{eq:minkowski_lower} to the
upper bound on $\mu_s(K,\Lambda)$ in Lemma \ref{lem:s-covering} leads
to
\begin{equation*}
\begin{split}
\mu_s(K,\Lambda) & \leq  \mu_1(K,\Lambda) +
(s-1)^\frac{1}{n}\left(\frac{\det\Lambda}{\vol K}\right)^\frac{1}{n}\\
&\leq \sum_{i=1}^n\lambda_i(K,\Lambda) +
(s-1)^\frac{1}{n}\left(n!\,\prod_{i=1}^n
  \lambda_i(K,\Lambda)\right)^\frac{1}{n}\\ &\leq
 \left(1+(n!^{1/n}/n) (s-1)^{1/n}\right)\sum_{i=1}^n\lambda_i(K,\Lambda),
\end{split}
\end{equation*}
by the arithmetic-geometric mean inequality.
\endproof

Unfortunately, we are not aware of a nice generalization of Jarnik's
lower bound (cf.~\cite[Lemma 2.4]{Kannan-Lovasz})
$\mu_1(K,\Lambda)\geq \lambda_n(K,\Lambda)$ to the $s$-covering radius.




\bigskip

\section{Frobenius number and covering radius}
\label{kannan}

For a given primitive positive vector $a=(a_1,\dots,a_n)^\intercal
\in\Z^n_{>0}$ let
\begin{equation*}
S_a = \left\{ x \in \R_{\geq 0}^{n-1} : a_1\,x_1+\cdots +a_{n-1}\,x_{n-1}\leq 1 \right\}
\end{equation*}
be the $(n-1)$-dimensional simplex with vertices
$0,\frac{1}{a_i}\,e_i$ where $e_i$ is the $i$-th unit vector in
$\R^{n-1}$, $1\leq i\leq n-1$. Furthermore, we consider the following
sublattice of $\Z^{n-1}$
\begin{equation*}
\Lambda_a = \left\{ z\in\Z^{n-1} : a_1\,z_1+\cdots + a_{n-1}\,z_{n-1}\equiv 0 \bmod a_n \right\}.
\end{equation*}
This simplex and lattice were introduced by Kannan in his studies of
the Frobenius number \cite{Kannan}, where he proved the following beautiful identity:
\begin{equation*}
       \mu_1(S_a,\Lambda_a)= \frob_1(a)+a_1+\cdots +a_n.
\end{equation*}
Here we just extend his arguments to the $s$-Frobenius number. We
start with the following lemma about an ``integral version'' of $\mu_s(S_a,\Lambda_a)$.

\begin{lemma} Let $n\geq 2$, $s\geq 1$, and let
$$\mu_s(S_a,\Lambda_a;\Z^{n-1})=\min \left\{ \rho>0 :
   \#\left\{(z+\rho\,S_a)\cap\Lambda_a\right\}\geq s\ \forall\ z\in\Z^{n-1} \right\}.$$
 Then
\begin{equation*}
  \mu_s(S_a,\Lambda_a;\Z^{n-1})= \frob_s(a)+a_n.
\end{equation*}
\label{lem:kannan_gen}
\end{lemma}
\begin{proof} To simplify the notation, for each $y\in\R^n$ let
$\tilde{y}=(y_1,\dots,y_{n-1})^\intercal$ be the vector consisting of the
first $(n-1)$ coordinates of $y$. Further, let
$\overline{\mu_s}=\mu_s(S_a,\Lambda_a;\Z^{n-1})$ and
$\frob_s=\frob_s(a)$.

First we show that $\overline{\mu_s}\leq \frob_s +a_n$. To this end,
let $z\in\Z^{n-1}$ and let $k\in\{1,\dots,a_n\}$ be such that
$\tilde{a}^\intercal\,z \equiv -(\frob_s+k) \bmod a_n$. By the
definition of $\frob_s$ we can find $b_1,\dots,b_s\in\Z^{n}_{\geq 0}$ with
$a^\intercal b_i  = \frob_s+k$, $1\leq i\leq s$. Hence we have found
$s$ different lattice vectors $z+\tilde{b_i}\in\Lambda_a$, $1\leq
i\leq s$, and since $\tilde{b_i}\in (\frob_s+k)\,S_a$ we obtain
\begin{equation*}
 z+\tilde{b_i} \in
 z+(\frob_s+a_n)\,S_a,\quad 1\leq i\leq s.
\end{equation*}
Hence $\overline{\mu_s}\leq \frob_s +a_n$, and it remains to show the
reverse inequality.

Since $\gcd(a)=1$, we can find a $z\in\Z^{n-1}$
with $\tilde{a}^\intercal z \equiv \frob_s\bmod a_n$. Now suppose
that for  a $0<\gamma<\frob_s+a_n$ we can find
$g_1,\dots,g_s\in\Lambda_a$ such that $g_i\in z+\gamma\,S_a$. Since
$\tilde{a}^\intercal (g_i-z)\equiv \frob_s\bmod a_n$ and
$\tilde{a}^\intercal (g_i-z)\leq \gamma<\frob_s+a_n$, we conclude that
there exist non-negative integers $m_i$ with
\begin{equation*}
         \tilde{a}^\intercal (g_i-z) = \frob_s- m_i\,a_n,\quad 1\leq
         i\leq s.
\end{equation*}
Since $g_i\in z+\gamma\,S_a$, we conclude that $(g_i-z)$ is a vector with
non-negative integer coordinates, and so
$\tilde{a}^\intercal (g_i-z)+ m_i\,a_n$, $1\leq i\leq s$, are $s$
different non-negative integral representations of
$\frob_s$, which contradicts the definition of $\frob_s$.
This proves that $\overline{\mu_s} \geq \frob_s +a_n$, and completes the
proof of the lemma.
\end{proof}

The next theorem is the canonical extension of  Kannan's Theorem 2.5
in \cite{Kannan} for the classical Frobenius number.

\begin{theorem} Let $n\geq 2$, $s\geq 1$. Then
\begin{equation*}
       \mu_s(S_a,\Lambda_a)= \frob_s(a)+a_1+\cdots +a_n.
\end{equation*}
\label{thm:kannan_gen}
\end{theorem}
\begin{proof}
We keep the notation of Lemma \ref{lem:kannan_gen} and its proof, and
in addition we set $\mu_s=\mu_s(S_a,\Lambda_a)$. In view of Lemma
\ref{lem:kannan_gen}, we have to show that
\begin{equation}
        \mu_s=\overline{\mu_s}+(a_1+\cdots +a_{n-1}).
\label{eq:toshow}
\end{equation}
First we verify the inequality $\mu_s\leq
\overline{\mu_s}+a_1+\cdots +a_{n-1}$. Since the
$(n-1)$-dimensional closed cube $\overline{P}=[0,1]^{n-1}$ of edge-length 1 is contained in
$(a_1+\cdots+a_{n-1})\,S_a$,
\begin{equation*}
\R^{n-1}=\Z^{n-1} + (a_1+\cdots+a_{n-1})\,S_a.
\end{equation*}
Hence, in view of  \eqref{eq:s-covering}, it suffices to verify
that for each $z\in\Z^{n-1}$
\begin{equation*}
    \#\left\{(z+\overline{\mu_s}\,S_a)\cap\Lambda_a\right\}\geq s,
\end{equation*}
which follows by the definition of $\overline{\mu_s}$.

Now suppose $\mu_s<  \overline{\mu_s}+a_1+\cdots +a_{n-1}$. By
Lemma \ref{lem:kannan_gen}, there exists a $z\in\Z^{n-1}$
such that for any subset $I_s\subset\Lambda_a$ of cardinality at least
$s$ there exists a $b\in I_s$ with
$(z-b)\notin\mathrm{int}(\overline{\mu_s}S_a)$, where
$\mathrm{int}(\,)$ denotes the interior of a set. Let $u \in \Z^{n-1}$ be the
vector with all coordinates equal to~1. By our assumption, there exists
at least $s$ lattice points $b_i\in \Lambda_a$, $1\leq i\leq s$, such
that
$(z+u) \in b_i+ \mathrm{int}((\overline{\mu_s}+a_1+\cdots +a_{n-1})S_a)$.
Then this is certainly also true for any sufficiently small positive
$\epsilon$ and the point $z+(1-\epsilon)u$. Thus,  for $1\leq i\leq s$,
\begin{equation*}
\begin{split}
\overline{\mu_s}+a_1+\cdots
+a_{n-1} &>\tilde{a}^\intercal(z+(1-\epsilon)u-b_i)\\ &= \tilde{a}^\intercal(z-b_i)
+ (1-\epsilon)(a_1+\dots+a_{n-1}).
\end{split}
\end{equation*}
Since $\epsilon$ is an arbitrary sufficiently small positive real number, we conclude that
$\tilde{a}^\intercal(z-b_i)<\overline{\mu_s}$, $1\leq i\leq s$. On the
other hand, we have  $z+(1-\epsilon)u-b_i\geq
0$, which implies $z-b_i\geq 0$, $1\leq i\leq s$. In other words, the
$s$ lattice points $b_1,\dots,b_s$ lie in the interior of
$z+\overline{\mu_s}\,S_a$ which contradicts the definition of~$\overline{\mu_s}$.
\end{proof}

We remark that in the case $n=2$, $S_a$ is just the segment
$[0,1/a_1]$ and $\Lambda_a$ is the set of all integral multiplies of $a_2$, i.e.,
$\Lambda_a= \Z\,a_2$. Hence, in this special case,
\begin{equation*}
    \mu_s(S_a,\Lambda_a)= s\,a_1\,a_2,
\end{equation*}
which gives, via Theorem \ref{thm:kannan_gen}, another proof of \eqref{eq:dim_two_ext}.

\begin{proof}[Proof of Theorem \ref{thm:main}] First we observe that
  $\det \Lambda_a=a_n$. This follows, for instance, from the fact that
  there are at most $a_n$ residue classes of the sublattice $\Lambda_a$ with respect
  to $\Z^{n-1}$, and since $\gcd(a)=1$ we have exactly $a_n$ distinct
  residue classes.  Next we note for  the ($(n-1)$-dimensional) volume
  of $S_a$ that
\begin{equation*}
 \vol(S_a)=\frac{1}{(n-1)!} \frac{1}{a_1\cdot\ldots\cdot a_{n-1}},
\end{equation*}
so that $\det\Lambda_a/\vol(S_a) = (n-1)!a_1\cdot\ldots\cdot
a_{n}$. Hence Lemma \ref{lem:s-covering} and Theorem
\ref{thm:kannan_gen} give the desired bounds.
\end{proof}
\bigskip

\section{Average behaviour}
\label{average}

It will be convenient to define
\begin{equation*}
 X_s(a)=\frac{\frob_s(a)}{(s\cdot a_1\,a_2\cdot\ldots\cdot
   a_n)^{1/(n-1)}}
\end{equation*}
for each $a\in \mathrm{G}(T)$. We start with the proof of  Corollary \ref{cor:main}.
\begin{proof}
For i), we observe that by \eqref{eq:prob_classic} we may assume $s\geq
2$ and by the upper bound of Theorem \ref{thm:main}
we have
\begin{equation}
 X_s(a) \leq
 s^{-\frac{1}{n-1}}X_1(a) + \mathrm{c_n},
\label{eq:average}
\end{equation}
for a dimensional constant $\mathrm{c}_n = \left( (n-1)! \right)^{\frac{1}{n-1}}$. Hence,
\eqref{eq:prob_classic} implies for, say, $D\geq 2\,\mathrm{c}_n$,
\begin{equation*}
\prob\left(
X_s(a)\geq D\right) \ll_n
\frac{1}{s}(D-\mathrm{c}_n)^{-(n-1)}\leq D^{-(n-1)}.
\end{equation*}
Now, in view of \eqref{eq:average_classic}, \eqref{eq:average} also implies that
\begin{equation*}
\frac{1}{\#\mathrm{G}(T)}\sum_{a\in \mathrm{G}(T)} X_s(a) \leq
s^{-\frac{1}{n-1}} \left(\frac{1}{\#\mathrm{G}(T)} \sum_{a\in \mathrm{G}(T)}
X_1(a) \right) + \mathrm{c}_n \ll_n 1.
\end{equation*}
In order to show that the left hand side is also bounded from below by
a constant depending only on $n$, we use the lower bound of Theorem \ref{thm:main} and obtain
\begin{equation*}
\frac{1}{\#\mathrm{G}(T)}\sum_{a\in \mathrm{G}(T)} X_s(a) >
\mathrm{c}_n-s^{-\frac{1}{n-1}} \frac{1}{\#\mathrm{G}(T)}\sum_{a\in
  \mathrm{G}(T)} \frac{a_1+\cdots +a_n}{(a_1\cdot\ldots\cdot a_n)^\frac{1}{n-1}}.
\end{equation*}
The latter sum has already been investigated in \cite{AHH}, where the proof of Proposition 1 shows precisely that
\begin{equation*}
\frac{1}{\#\mathrm{G}(T)}\sum_{a\in
  \mathrm{G}(T)} \frac{a_1+\cdots +a_n}{(a_1\cdot\ldots\cdot a_n)^\frac{1}{n-1}}\leq\mathrm{C}_n\,T^{-\frac{1}{n-1}},
\end{equation*}
for another constant $\mathrm{C}_n$ depending only on $n$. Hence, for
sufficiently large $T$ we obtain
$$\frac{1}{\#\mathrm{G}(T)}\sum_{a\in
  \mathrm{G}(T)} X_s(a)\gg_n 1,$$
which completes the proof of ii).
\end{proof}


\noindent
{\it Acknowledgement.} We would like to thank Matthias Henze, Eva Linke and
Carsten Thiel for helpful comments.

\end{document}